\documentclass[11pt]{amsart}
\usepackage{amssymb}
\usepackage{amsmath}
\usepackage{amsthm}
\usepackage{mathrsfs}
\usepackage[all]{xy}
\usepackage{bbm}
\usepackage{fullpage}

\newcommand{\bB}{\ensuremath{\mathbb{B}}}

\newcommand{\sC}{\ensuremath{\mathscr{C}}}
\newcommand{\sD}{\ensuremath{\mathscr{D}}}

\newcommand{\bK}{\ensuremath{\mathbb{K}}}

\newcommand{\cP}{\ensuremath{\mathcal{P}}}
\newcommand{\sP}{\ensuremath{\mathscr{P}}}

\newcommand{\sV}{\ensuremath{\mathscr{V}}}

\newcommand{\smsh}{\wedge}
\newcommand{\iso}{\cong}
\newcommand{\into}{\hookrightarrow}

\newcommand{\boxB}{\boxed{\mathbf{B}}}

\newcommand{\weakName}{{\bf SMC}-category}
\newcommand{\weakNames}{{\bf SMC}-categories}
\newcommand{\WeakNames}{{\bf SMC}-categories}

\newcommand{\weakMathName}{\ensuremath{\mathbf{SMC\text{-}Cat}}}
\newcommand{\strictName}{{\bf PC}-category}
\newcommand{\strictNames}{{\bf PC}-categories}
\newcommand{\weakFunctor}{\ensuremath{\mathbf{SMC}}-functor}
\newcommand{\weakFunctors}{{\weakFunctor}s}

\newcounter{myCounter}

\begin{document}
\title[Strictification of \weakNames]{
Strictification of categories weakly enriched in symmetric monoidal categories.
}
\author{
Bertrand Guillou}
\address{Department of Mathematics, University of Illinois at Urbana-Champaign, 1409 W. Green Street, Urbana, IL 61801}
\email{bertg@illinois.edu}


\begin{abstract}We offer two proofs that categories weakly enriched over symmetric monoidal categories can be strictified to categories enriched in permutative categories. This is a "many $0$-cells" version of the strictification of bimonoidal categories to strict ones.
\end{abstract}
\subjclass[2000]{Primary 19D23; Secondary 18D05}
\keywords{Bimonoidal Category, Coherence}

\maketitle

\numberwithin{equation}{section}
\numberwithin{figure}{section}

\newtheorem{thm}{Theorem}[section]
\newtheorem{lemma}[thm]{Lemma}
\newtheorem{conj}[thm]{Conjecture}
\newtheorem{prop}[thm]{Proposition}
\newtheorem{cor}[thm]{Corollary}

\theoremstyle{definition}
\newtheorem{qst}{Question}
\newtheorem{const}[thm]{Construction}
\newtheorem{defn}[thm]{Definition}
\newtheorem{eg}[thm]{Example}
\newtheorem{rmk}[thm]{Remark}


\section{Introduction}

Categories with additional structure have long played an important role in homotopy theory. The classifying space $B\sC$ of a symmetric monoidal category $\sC$ is an $H$-space which is associative and commutative up to all possible higher homotopies, alias an $E_\infty$ space. Up to a group completion, $B\sC$ is thus the zeroth space of a spectrum, called the $\bK$-theory spectrum of $\sC$.
By a classical result of Isbell \cite{Isb}, any (symmetric) monoidal category is equivalent to one in which the monoidal product is strictly associative and unital, and it is convenient to replace a symmetric monoidal $\sC$ by a strict one in order to build the spectrum $\bK(\sC)$.

In cases of interest, such as the categories of finite sets or finitely generated projective modules over a commutative ring, the symmetric monoidal category $\sC$ has a second monoidal structure, which is to be thought of as multiplicative. The resulting structure is called a bimonoidal category, and again any bimonoidal category is equivalent to a strict one \cite[\S6.3]{MQRT}. Using a $\bK$-theory functor that has good multiplicative properties, as in \cite{EM}, the $\bK$-theory spectrum $\bK(\sC)$ of a (strict) bimonoidal category $\sC$ inherits the structure of a ring spectrum. In the case that the multiplicative monoidal structure is symmetric, $\bK(\sC)$ is a commutative ring spectrum.

It is then natural to ask for structure on a (strict) symmetric monoidal $\sD$ that will make $\bK(\sD)$ a module over $\bK(\sC)$ for (strict) bimonoidal $\sC$; this question has been studied in \cite{EM}, even in the more difficult case of a symmetric bimonoidal $\sC$. 

There is another way to think about rings and modules, through the language of enriched categories. Recall that a (small) spectral category $\bB$ has a set of objects $\{a, b, \dots\}$. For every pair of objects $a$ and $b$, there is a spectrum $\bB(a,b)$, and for every triple of objects $a$, $b$, $c$ one has a composition pairing
\[\bB(b,c)\smsh\bB(a,b)\longrightarrow\bB(a,c).\]
Finally, for every $a$ there is a unit map
\[S^0\longrightarrow\bB(a,a),\]
and the composition is associative and unital. Thus, each $\bB(a,a)$ is a ring spectrum, and each $\bB(a,b)$ is a $\bB(b,b)$-$\bB(a,a)$-bimodule. One may ask what categorical structure, when fed into a nice $\bK$-theory machine, will produce a spectral category, and again an answer is provided by the work of \cite{EM}. Essentially, the categorical input is a $2$-category $\mathbf{C}$, together with a strict symmetric monoidal structure on each $\mathbf{C}(c,c')$, such that the composition maps are bilinear. 
This is essentially a category enriched in the category of permutative categories (i.e. symmetric strict monoidal categories). Following the convention of  writing $\sV$-category for ``category enriched in $\sV$'', we call such a structure a PermCat-category, or \strictName\ for short, and the work of \cite{EM} provides the following result:

\begin{thm} If $\mathbf{C}$ is a \strictName, then $\bK(\mathbf{C})$ is a spectral category.
\end{thm}

The structures that one encounters in practice, however, are not typically \strictNames\ but rather the less strict version that we call weak SymMonCat-categories, or \weakNames\ for short. Rather than a $2$-category one has a weak $2$-category, or bicategory, and the monoidal structure on each $\mathbf{C}(c,c')$ is not strict. 
%
As one might expect, a weak structure of this sort can be rigidified into a strict structure, and this is our main result.

\begin{thm}\label{MainThm} Any \weakName\ is biequivalent to a \strictName\ via a map of \weakNames.
\end{thm}

Thus the structures that arise in nature can be suitably perturbed to strutures that feed naturally into a $\bK$-theory machine to produce a spectral category. We offer two proofs of Theorem~\ref{MainThm}. The first, more explicit, argument is given in \S\ref{CohVerISect} and generalizes in a straightforward way the arguments of \cite{Isb} and \cite{MQRT}. The second proof, given in \S\ref{CohVerIISect}, follows the Yoneda approach to coherence (\cite{Lein}). The final \S\ref{SingZeroSect} discusses the strictification of \S\ref{CohVerIISect} for bimonoidal categories.

%
%
%
%

We use the language of bicategories throughout and suggest \cite{Lein} as a quick introduction to the relevant terminology. We would like to thank Mike Shulman for a number of helpful comments.

\section{\WeakNames}

In this section, we introduce the main object of study and give a number of examples.

\begin{defn}\label{weakNameDefn} An {\bf \weakName} is a bicategory $\mathbf B$ together with, for each pair $a,b$ of $0$-cells, a symmetric monoidal structure on the category ${\mathbf B}(a,b)$. These monoidal structures must be compatible in the following sense: for each triple $a,b,c$ of $0$-cells, we ask that the composition functors 
\[\circ:{\mathbf B}(b,c)\times{\mathbf B}(a,b)\to{\mathbf B}(a,c)\] 
be bilinear, in the sense that the adjoints
\[{\mathbf B}(b,c)\xrightarrow{-\circ(-)}\mathbf{CAT}({\mathbf B}(a,b),{\mathbf B}(a,c))\]
and
\[{\mathbf B}(a,b)\xrightarrow{(-)\circ-}\mathbf{CAT}({\mathbf B}(b,c),{\mathbf B}(a,c))\]
are equipped with symmetric strong monoidal structures (the functor categories inherit monoidal structures from ${\mathbf B}(a,c)$).
These conditions encode natural distributivity isomorphisms as well as natural isomorphisms exhibiting the monoidal unit objects as null objects for composition. Finally, the above data must satisfy the following conditions:
\begin{enumerate}
\item\label{bilinCond} The isomorphisms $(f\oplus g)\circ - \iso (f\circ -)\oplus (g\circ -)$, $-\circ(f\oplus g)\iso (-\circ f)\oplus (-\circ g)$, $\mathbf{0}\circ-\iso \mathbf{0}$, and $-\circ\mathbf{0}\iso\mathbf{0}$ are isomorphisms of monoidal functors
\item\label{bicatMndCond} The isomorphisms $f\circ(g\,\circ-)\iso(f\circ g)\circ-$, $(-\circ f)\circ g\iso -\circ(f\circ g)$, $f\circ(-\circ g)\iso (f\circ -)\circ g$, $-\circ\mathbf{1}\iso\mathrm{id}$, and $\mathbf{1}\circ -\iso\mathrm{id}$ are isomorphisms of monoidal functors.
\end{enumerate}
\end{defn}

\begin{rmk} The above definition can alternatively be given as follows: an \weakName\ is a bicategory $\mathbf{B}$ with symmetric monoidal structures on each $\mathbf{B}(a,b)$ such that 
\begin{enumerate}
\item The adjoints to the composition functors factor through symmetric strong monoidal functors 
\[\mathbf{B}(b,c)\to\mathbf{SymMon}(\mathbf{B}(a,b),\mathbf{B}(a,c))\]
and
\[\mathbf{B}(a,b)\to\mathbf{SymMon}(\mathbf{B}(b,c),\mathbf{B}(a,c)),\]
where $\mathbf{SymMon}(\mathbf{C},\mathbf{D})$ denotes the category of symmetric strong monoidal functors and monoidal transformations. The above specifies two monoidal structures on each of the functors $f\circ(-)$ and $(-)\circ f$, and these are required to coincide.
\item The associativitiy and unit constraints for the bicategory structure are monoidal transformations, as in condition~(\ref{bicatMndCond}) of Definition~\ref{weakNameDefn}.
\end{enumerate}
\end{rmk}

\begin{eg}\label{ringEG} 
Any ordinary ring or semiring $R$, considered as a $2$-category with a single $0$-cell and only identity $2$-cells, is an \weakName.
\end{eg}

\begin{eg} In general, \weakNames\ with a single $0$-cell are bimonoidal categories, though not in the sense of Laplaza (\cite{Lap}). Laplaza takes the multiplicative monoidal structure to be symmetric, so we shall refer to Laplaza's bimonoidal categories as symmetric bimonoidal categories, and our bimonoidal categories do not in general give examples of symmetric bimonoidal categories. On the other hand, Laplaza does not take his distributivity maps to be isomorphisms, so Laplaza's symmetric bimonoidal categories do not generally give bimonoidal categories in our sense. 

Symmetric bimonoidal categories as redefined in \cite[\S6.3]{MQRT} do give bimonoidal categories in our sense.
\end{eg}

\begin{eg} Let $\sC$ be any category with finite coproducts and fiber products. Suppose that pullbacks in $\sC$ preserve coproducts, as is the case in the category of sets or $G$-sets. Then the bicategory ${\mathbf B}=Span(\sC)$ of spans in $\sC$ is a typical example of an \weakName. The bicategory structure comes from pullbacks, and the additional monoidal structure comes from coproducts. 
\end{eg}

\begin{eg} Let $\mathbf{Mod}$ denote the bicategory whose $0$-cells are rings, whose $1$-cells are bimodules, and whose $2$-cells are bimodule maps. The horizontal composition is given by tensor products. Since tensor products preserve sums, this is an example of an \weakName.
\end{eg}

\begin{eg} Let $\mathbf{SymMon}_{}$ denote the $2$-category of symmetric monoidal categories, symmetric strong monoidal functors, and monoidal transformations. For symmetric monoidal categories $\sC$ and $\sD$, the functor category $\mathbf{SymMon}_{}(\sC,\sD)$ inherits a monoidal structure from $\sD$: the sum of strong monoidal functors $F,G:\sC\to \sD$ is defined on $c\in\sC$ by
\[(F\oplus G)(c) =F(c)\oplus G(c),  \]
and the unit object is the constant functor at the unit of $\sD$.
Note that $F\oplus G$ is a symmetric strong monoidal functor: for instance, the structure morphism is given by
\[ \begin{split}
(F\oplus G)(c_1) \oplus (F\oplus G)(c_2) &= \bigl(F(c_1)\oplus G(c_1)\bigr) \oplus \bigl( F(c_2)\oplus G(c_2) \bigr) \\
 (assoc.)\qquad\qquad  & \iso F(c_1)\oplus \biggl( \bigl( G(c_1)\oplus F(c_2) \bigr) \oplus G(c_2)\biggr)  \\
 (comm.)\qquad\qquad & \iso F(c_1)\oplus \biggl( \bigl( F(c_2)\oplus G(c_1) \bigr) \oplus G(c_2)\biggr)  \\
 (assoc.)\qquad\qquad & \iso \bigl(F(c_1)\oplus F(c_2)\bigr) \oplus \bigl( G(c_1)\oplus G(c_2) \bigr) \\
 (F, G \text{ strong monoidal})\quad 
 & \iso F(c_1\oplus c_2)\oplus G(c_1\oplus c_2) = (F\oplus G)(c_1\oplus c_2) .
\end{split} \]
This structure makes $\mathbf{SymMon}_{}$ into an \weakName.
\end{eg}

\begin{eg}\label{PermEG} Let $\mathbf{Perm}_{}$ denote the full sub-$2$-category of $\mathbf{SymMon}_{}$ consisting of permutative categories, symmetric strong monoidal functors, and monoidal transformations. This is again an \weakName. Note that for each pair $\sP_1$, $\sP_2$ of permutative categories, the category $\mathbf{Perm}_{}(\sP_1,\sP_2)$ is permutative.
\end{eg}

\begin{eg}\label{PermuEG} Let $\mathbf{Perm}_{u}$ denote the locally full sub-$2$-category of $\mathbf{Perm}_{}$ consisting of permutative categories, symmetric strong monoidal functors that are strictly unital, and monoidal transformations. This is again an \weakName. The essential point is that if $F,G:\sP_1\to\sP_2$ are strictly unital, symmetric strong monoidal functors, then so is $F\oplus G$:
\[(F\oplus G)(u_1)=F(u_1)\oplus G(u_1)=u_2\oplus u_2=u_2\]
 since $\sP_2$ is permutative.
\end{eg}

\begin{rmk} One can also consider the sub-$2$-category $\mathbf{Perm}_{\mathrm{strict}}\subset\mathbf{Perm}_{}$ consisting of permutative categories, symmetric strict monoidal functors, and monoidal transformations. This $2$-category, however, does not give an example of an \weakName \ since there is no canonical monoidal structure on the categories $\mathbf{Perm}_{\mathrm{strict}}(\sP_1,\sP_2)$: the sum of symmetric strict monoidal functors is only symmetric {\em strong} monoidal in general, as in order to identify
\[(F\oplus G)(x)\oplus (F\oplus G)(y)=F(x)\oplus G(x)\oplus F(y)\oplus G(y)\]
with
\[(F\oplus G)(x\oplus y)=F(x\oplus y)\oplus G(x\oplus y)=F(x)\oplus F(y)\oplus G(x)\oplus G(y)\]
one must use a commutativity isomorphism.

On the other hand, versions of the above examples using lax monoidal functors rather than strong monoidal functors also produce \weakNames.
\end{rmk}

We close this section with a construction that will be needed in \S\ref{CohVerIISect}. Recall that if $\mathbf{B}$ is a bicategory, there is an opposite bicategory $\mathbf{B}^{op}$ in which the composition of $1$-cells is reversed. The $0$-cells of $\mathbf{B}^{op}$ are those of $\mathbf{B}$, and 
\[\mathbf{B}^{op}(a,b)=\mathbf{B}(b,a).\]
Composition is defined using the isomorphism of categories
\[ \mathbf{B}(c,b)\times\mathbf{B}(b,a)\iso\mathbf{B}(b,a)\times\mathbf{B}(c,b).\]

\begin{prop}
If $\mathbf{B}$ has the structure of an \weakName, the bicategory $\mathbf{B}^{op}$ inherits the structure of \weakName.
\end{prop}

\section{\strictNames}

Here we introduce the strict versions of \weakNames\ and give a few examples.

\begin{defn} A {\bf \strictName} is a $2$-category $\mathbf C$ together with, for each pair $c,d$ of $0$-cells, a structure of permutative category on ${\mathbf C}(c,d)$. The functors 
\[{\mathbf C}(b,c)\xrightarrow{-\circ(-)}\mathbf{CAT}({\mathbf C}(a,b),{\mathbf C}(a,c))\]
and
\[{\mathbf C}(a,b)\xrightarrow{(-)\circ-}\mathbf{CAT}({\mathbf C}(b,c),{\mathbf C}(a,c))\]
are now functors between permutative categories, and we require a symmetric strict monoidal structure on the functor $-\circ(-)$ and a symmetric strong monoidal, but strictly unital, structure on the functor $(-)\circ-$. The coherence conditions (\ref{bilinCond}) and (\ref{bicatMndCond}) of Definition~\ref{weakNameDefn} are again required to hold.
\end{defn}

Since many of the structure maps are now identities, a number of the coherence conditions hold automatically. Letting $\delta:f\circ(g\oplus h)\iso f\circ g\oplus f\circ h$ be the left distributivity isomorphism, the conditions that are not automatic are as follows:
\renewcommand{\themyCounter}{\roman{myCounter}}  
\begin{list}{(\roman{myCounter})}{\usecounter{myCounter}}
\item\label{PentCond} For any $f_1, f_2, g_1$, and $g_2$, the following diagram commutes
\[\xymatrix  @C=-7em  {
(f_1 \oplus f_2)\circ (g_1 \oplus g_2) \ar@{=}[rr] \ar[d]^{\delta} & & (f_1\circ(g_1\oplus g_2))\oplus (f_2\circ(g_1\oplus g_2)) \ar[d]^{\delta\oplus\delta} \\
\bigl((f_1\oplus f_2)\circ g_1\bigr)\oplus \bigl((f_1\oplus f_2)\circ g_2\bigr) \ar@{=}[dr] & &  (f_1\circ g_1)\oplus (f_1\circ g_2) \oplus (f_2\circ g_1)\oplus (f_2\circ g_2) \ar[dl]^{1\oplus\gamma\oplus1} \\
 & (f_1\circ g_1)\oplus (f_2\circ g_1)\oplus (f_1\circ g_2)\oplus (f_2\circ g_2) 
}\]
\item For any $f$ and $g$, the isomorphism $\delta:\mathbf{0}\circ(f\oplus g)\iso \mathbf{0}\circ f \oplus \mathbf{0}\circ g$ is the idenity map of $\mathbf{0}$
\item For any $f$, $g$, $h_1$, and $h_2$, the following diagram commutes
\[\xymatrix{ f\circ g\circ (h_1\oplus h_2) \ar[rr]^\delta \ar[dr]_{1\cdot \delta} & & f\circ g\circ h_1\oplus f\circ g\circ h_2 \\  & f\circ (g\circ h_1\oplus g\circ h_2) \ar[ur]_{\delta}
}\]
\item For any $f$, $g_1$, $g_2$, and $h$, the following diagram commutes
\[\xymatrix{ f\circ (g_1\oplus g_2)\circ h \ar@{=}[r] \ar[d]_{\delta\cdot 1} & f\circ (g_1\circ h \oplus g_2\circ h) \ar[d]^\delta \\  (f \circ g_1\oplus f\circ g_2)\circ h \ar@{=}[r] & f\circ g_1\circ h \oplus f\circ g_2\circ h
}\]
\item For any $f$ and $g$, the isomorphism $\delta:\mathbf{1}\circ (f\oplus g)\iso \mathbf{1}\circ f\oplus \mathbf{1}\circ g$ is the identity map of $f\oplus g$.
\end{list}

As in the theory of bipermutative categories, the diagram in condition (\ref{PentCond}) above shows that it is unreasonable to take both distributivity conditions to hold strictly unless  commutativity also holds strictly. 

\begin{eg} Any ring $R$, considered as a $2$-category as in Example~\ref{ringEG}, gives an example of a \strictName. Indeed, the only $2$-cells are identity maps, and so all conditions are trivially satisfied.
\end{eg}

\begin{eg} Bipermutative categories, as specified originally in \cite{MQRT}, are examples of single $0$-cell \strictNames, although the multiplicative structure is assumed to be commutative. The version of bipermutative categories specified in \cite[Definition~3.6]{EM} do not give single $0$-cell \strictNames\ as those authors do not require the distributivity maps to be isomorphisms.  The ring categories of \cite[Definition~3.3]{EM} are closer to singe $0$-cell \strictNames, although again the authors do not take the distributivity maps to be isomorphisms. Single $0$-cell \strictNames\ thus give examples of ring categories, but not conversely.
\end{eg}

\begin{eg} The $2$-category $\mathbf{Perm}_{u}$ of Example~\ref{PermEG} is a \strictName. For strict monoidal functors $\xymatrix {\sP_1 \ar[r]^F & \sP_2 \ar@<.5ex>[r]^G \ar@<-.5ex>[r]_H & \sP_3  }$, we have
\[((H\oplus G)\circ F)(x)=HF(x)\oplus GF(x)=(HF\oplus HG)(x)\]
and
\[(u_3\circ F)(x)=u_3,\]
so $-\circ(-)$ is strict symmetric monoidal. On the other hand, given
$\xymatrix{ \sP_1 \ar@<.5ex>[r]^F \ar@<-.5ex>[r]_G & \sP_2 \ar[r]^H & \sP_3}$, we have
\[(H\circ(F\oplus G))(x)=H(F(x)\oplus G(x))\iso HF(x)\oplus HG(x)=(HF\oplus HG)(x)\]
and
\[(H\circ u_2)(x)=H(u_2)=u_3,\]
so $(-)\circ -$ is strictly unital and strong monoidal. The coherence conditions are easily verified.
Although the choice of which distributivity law to make strict is arbitrary, this example provides a good reason for the choice taken here (and in \cite{MQRT}).
\end{eg}

\section{The $3$-category of \weakNames}

In order to give the approach to coherence via the Yoneda embedding, we must first describe the $3$-category of \weakNames.

\begin{defn} Let $\mathbf{B}$ and $\mathbf{C}$ be \weakNames. An {\bf \weakFunctor} $F:\mathbf{B}\to\mathbf{C}$ is a (strong) functor\footnote{These also go by the name of homomorphism or pseudo-functor.} of bicategories such that each functor $F:\mathbf{B}(a,b)\to\mathbf{C}(F(a),F(b))$ is (strong) symmetric monoidal. Moreover, the natural transformations

\centerline{ \xymatrix {
 \mathbf{B}(b,c) \ar[r]^(.33){-\circ(-)} \ar[d]_F & \mathbf{CAT}(\mathbf{B}(a,b),\mathbf{B}(a,c)) \ar[r]^(.45){F\circ} \ar@{}[d]|{\Downarrow} & \mathbf{CAT}(\mathbf{B}(a,b),\mathbf{C}(Fa,Fc)) \\
 \mathbf{C}(Fb,Fc) \ar[rr]_(.43){-\circ(-)}  & &  \mathbf{CAT}(\mathbf{C}(Fa,Fb), \mathbf{C}(Fa,Fc)) \ar[u]_{\circ F}
}}
\noindent and

\centerline{ \xymatrix {
 \mathbf{B}(a,b) \ar[r]^(.33){(-)\circ-} \ar[d]_F & \mathbf{CAT}(\mathbf{B}(b,c),\mathbf{B}(a,c)) \ar[r]^(.45){F\circ} \ar@{}[d]|{\Downarrow} & \mathbf{CAT}(\mathbf{B}(b,c),\mathbf{C}(Fa,Fc)) \\
 \mathbf{C}(Fa,Fb) \ar[rr]_(.43){-\circ(-)}  & &  \mathbf{CAT}(\mathbf{C}(Fb,Fc), \mathbf{C}(Fa,Fc)) \ar[u]_{\circ F}
}}
\noindent are required to be monoidal transformations.
\end{defn}

\begin{defn} Given \weakNames\ $\mathbf{B}$ and $\mathbf{C}$ and \weakFunctors\ $F,G:\mathbf{B}\to\mathbf{C}$, a (strong) {\bf monoidal transformation} $\eta:F\Rightarrow G$ is a (strong) transformation in the sense of bicategories such that for each $a,b\in\mathbf{B}$, the natural transformation

\centerline{ \xymatrix @R=2ex {
 & \mathbf{C}(Ga,Gb) \ar[dr]^{\circ \eta_a} \ar@{}[dd]|{\Downarrow}_\eta \\
 \mathbf{B}(a,b) \ar[ur]^{G} \ar[dr]_F & & \mathbf{C}(Fa,Gb) \\
  & \mathbf{C}(Fa,Fb) \ar[ur]_{\eta_b\circ}
}}
\noindent is a monoidal transformation in the usual sense.
\end{defn}

\begin{eg} Let $R$ and $S$ be rings, considered as \weakNames\ as in Example~\ref{ringEG}. Then \weakFunctors\ $F:R\to S$ correspond to ring homomorphisms. Given two such $F$ and $G$, a monoidal transformation $\eta:F\Rightarrow G$ is given by an element $s\in S$ such that $sF(r)=G(r)s$ for all $r\in R$. In particular, monoidal transformations $F\Rightarrow F$ correspond to centralizers of $F(R)$ in $S$.
\end{eg}

\begin{defn}
Given \weakNames\ $\mathbf{B}$ and $\mathbf{C}$, \weakFunctors\ $F,G:\mathbf{B}\to\mathbf{C}$, and monoidal transformations $\eta,\sigma:F\Rightarrow G$, a {\bf modification} $M:\eta\Rrightarrow\sigma$ is simply a modification in the sense of bicategories.
\end{defn} 

\begin{defn} Given \weakNames\ $\mathbf{B}$ and $\mathbf{C}$, denote by $\weakMathName(\mathbf{B},\mathbf{C})$ the bicategory of \weakFunctors, strong monoidal transformations, and modifications.
\end{defn}

For any \weakNames\ $\mathbf{B}$ and $\mathbf{C}$, there is a canonical functor of bicategories 
\begin{equation}\weakMathName(\mathbf{B},\mathbf{C})\longrightarrow \mathbf{Bicat}(\mathbf{B},\mathbf{C}),\label{LocalEquiv}\end{equation}
which is locally full and faithful.

\begin{defn} Let $\mathbf{B}$ and $\mathbf{C}$ be \weakNames. A {\bf biequivalence} from $\mathbf{B}$ to $\mathbf{C}$ consists of a pair of \weakFunctors\ $F:\mathbf{B}\rightleftarrows\mathbf{C}:G$ together with an equivalence $FG\simeq 1_{\mathbf{C}}$ in $\weakMathName(\mathbf{C},\mathbf{C})$ and an equivalence $1_{\mathbf{B}}\simeq GF$ in $\weakMathName(\mathbf{B},\mathbf{B})$.
\end{defn}

As one might expect, there is the following alternative characterization of biequivalences.

\begin{prop} An \weakFunctor\ $F:\mathbf{B}\to\mathbf{C}$ of \weakNames\ is a biequivalence if and only if it is a local equivalence and is surjective-up-to-equivalence on $0$-cells, meaning that every $0$-cell in $\mathbf{C}$ is equivalent to some $0$-cell $F(b)$.
\end{prop}

\section{Coherence for \weakNames, version $I$}\label{CohVerISect}


The first proof of Theorem~\ref{MainThm} we offer is 
an analogue of Isbell's argument (\cite{Isb}). This is just a many-$0$-cell version of the argument in Proposition~VI.3.5 of \cite{MQRT}, dropping multiplicative commutativity.


\begin{proof}[First proof of Theorem~\ref{MainThm}] 
Let $\mathbf B$ be an \weakName. The first step is to rigidify the multiplicative structure, using the standard rigidification of bicategories to $2$-categories. Define a $2$-category $\mathbf{B}'$ to have the same $0$-cells as those of $\mathbf B$. The $1$-cells in $\mathbf{B}'(a,b)$ are defined to be formal strings 
\[ f_n\cdot f_{n-1}\cdot\dots\cdot f_1 \]
of composable $1$-cells in $\mathbf{B}(a,b)$ such that the source of $f_1$ is $a$ and the target of $f_n$ is $b$. In the case that $a=b$, we allow empty strings as well, denoted $1_a$, to serve as strict units for horizontal composition. There is a surjective function $\pi$ from the $1$-cells of $\mathbf{B}'(a,b)$ to those of $\mathbf{B}(a,b)$ defined by
\[ \pi(f_n\cdot f_{n-1}\cdot \dots\cdot f_1) = f_n\circ ( f_{n-1}\circ (f_{n-2}\circ(\dots\circ(f_2\circ f_1)\dots))),  \qquad \pi(1_a) = \mathbbm{1}_a. \]
One then defines the $2$-cells in $\mathbf{B}'$ so that this function extends to an equivalence 
\[\pi:\mathbf{B}'(a,b)\xrightarrow{\sim}\mathbf{B}(a,b).\]
Concatenation of strings makes $\mathbf{B}'$ into a $2$-category, and $\pi:\mathbf{B}'\to\mathbf{B}$ is a biequivalence. Note that there is also a canonical functor of bicategories $\eta:\mathbf{B}\to\mathbf{B}'$ which is the identity on $0$-cells and sends a $1$-cell of $\mathbf{B}$ to the singleton string in $\mathbf{B}'$. Then $\pi\eta=1_{\mathbf{B}}$ and $\eta\pi\simeq1_{\mathbf{B}'}$, so that $\eta$ is a quasi-inverse to $\pi$.

Each category $\mathbf{B}'(a,b)$ inherits a symmetric monoidal structure from that in $\mathbf{B}(a,b)$. Given a pair $f,g\in\mathbf{B}'(a,b)$ of $1$-cells, we simply define $f\oplus g:=\pi(f)\oplus\pi(g)$, considered as a singleton string. Moreover, $\mathbf{B}'$ becomes an \weakName, and $\eta$ is a biequivalence of \weakNames.

We can now perform Isbell's rigidification construction to each monoidal category $\mathbf{B}'(a,b)$. Call the resulting permutative category $\boxed{\mathbf{B}}(a,b)$.
A typical object of $\boxB(a,b)$ is a (possibly empty) formal sum
\[f_{1n_1}\cdots f_{11}+f_{2n_2}\cdots f_{21}+\dots+f_{kn_k}\cdots f_{k1}.\]
Again, there is a surjective function $\boxed{\pi}$ from the $1$-cells of $\boxed{\mathbf{B}}(a,b)$ to the $1$-cells of $\mathbf{B}'(a,b)$ defined by
\[\boxed{\pi}(\overline{g_1}+\overline{g_2}+\dots+\overline{g_n}) = \overline{g_1}\oplus(\overline{g_2}\oplus\dots\oplus(\overline{g_{n-1}}\oplus\overline{g_n})\dots), \qquad \boxed{\pi}({_a}\emptyset_b)={_a}u_b \]
for $\overline{g_i}\in\mathbf{B}'(a,b)$,
and this is used to define the $2$-cells of $\boxed{\mathbf{B}}(a,b)$. The function $\boxed{\pi}$ then extends to a strong symmetric monoidal functor $\boxed{\pi}:\boxed{\mathbf{B}}(a,b)\to\mathbf{B}'(a,b)$.

We claim that $\boxB$ is a \strictName \ (the $0$-cells are the same as those of $\mathbf{B}$ and $\mathbf{B}'$). By construction, each $\boxB(a,b)$ is a permutative category. The composition functor
\[\boxed{\circ}:\boxB(b,c)\times\boxB(a,b)\to\boxB(a,c)\]
is defined on $1$-cells by
\[\begin{split}
 &(f_{1n_1}\cdots f_{11}+f_{2n_2}\cdots f_{21}+\dots+f_{kn_k}\cdots f_{k1})  \quad \boxed{\circ} \\
& \qquad \qquad (g_{1m_1}\cdots g_{11}+g_{2m_2}\cdots g_{21}+\dots+g_{lm_l}\cdots g_{l1}) :=   \\
&  f_{1n_1}\cdots f_{11} g_{1m_1}\cdots g_{11} + f_{1n_1}\cdots f_{11} g_{2m_2}\cdots g_{21} + \dots + f_{kn_k}\cdots f_{k1} g_{lm_l}\cdots g_{l1}, \\
& \qquad\qquad\qquad\qquad {_b}\emptyset_c \ \boxed{\circ} \ \alpha := {_a}\emptyset_c, \qquad\qquad\qquad \beta \ \boxed{\circ} \ {_b}\emptyset_a :={_c}\emptyset_a .
\end{split}\]
Finally, the isomorphisms 
\[\boxed{\pi}\bigl( (\overline{g_1}+\dots+\overline{g_n})\boxed{\circ}(\overline{f_1}+\dots+\overline{f_k}) \bigr)\iso \boxed{\pi}(\overline{g_1}+\dots+\overline{g_n})\circ \boxed{\pi}( \overline{f_1}+\dots+\overline{f_k}) \]
are used to define $\boxed{\circ}$ on the level of $2$-cells (if one of the source or target $1$-cells is of the form $\emptyset$, then the nullity isomorphism ${_b}u_c\circ\pi(\overline{f})\iso{_a}u_c$ must be used).

We leave it to the reader to verify the desired properties of the composition functors $-\circ(-)$ and $(-)\circ -$ but mention that the failure of left distributivity to be strict comes from the use of a commutativity isomorphism in $\boxed{\mathbf{B}}(a,b)$.

The functor $\boxed{\pi}:\boxB\to\mathbf{B}'$ is a biequivalence of \weakNames\ with quasi-inverse $\boxed{\eta}$, where $\boxed{\eta}$ is the identity on $0$-cells and sends a $1$-cell to itself, considered as a singleton sum. We then have biequivalences of \weakNames\ 
\[ \xymatrix{
\mathbf{B} \ar@<.4ex>[r]^{\eta} & \mathbf{B}' \ar@<0.4ex>[l]^{\pi} \ar@<0.4ex>[r]^{\boxed{\eta}} & \boxB \ar@<0.4ex>[l]^{\boxed{\pi}}.
}\]
\end{proof}

\section{Coherence for symmetric monoidal categories}




The classical strictification of (symmetric) monoidal categories fits into our picture as the following result, which will be needed later.

\begin{prop} The inclusion \weakFunctor\ $\iota:\mathbf{Perm}_{}\into\mathbf{SymMon}_{}$ is a biequivalence of \weakNames.
\end{prop}

\begin{proof} 
Given $\sP_1,\sP_2\in\mathbf{Perm}$, we have $\mathbf{Perm}(\sP_1,\sP_2)=\mathbf{SymMon}(\iota\sP_1,\iota\sP_2)$, so it remains to show that every $\sC\in\mathbf{SymMon}$ is equivalent (in $\mathbf{SymMon}$) to some $\iota(\sP)$. But this is precisely Isbell's result that every symmetric monoidal category is monoidally equivalent to a permutative category.

\renewcommand{\qedsymbol}{$\blacksquare$}
\end{proof}


\begin{prop} The inclusion \weakFunctor\ $\mathbf{Perm}_{u}\into\mathbf{Perm}_{}$ is a biequivalence of \weakNames.
\end{prop}

\begin{proof} The inclusion is locally full and faithful and is a bijection on $0$-cells, so it remains only to show that it is locally essentially surjective. That is, given a strong symmetric monoidal functor $\Phi:\sP_1\to\sP_2$, we must provide a strictly unital strong monoidal functor $\Psi:\sP_1\to\sP_2$ and a monoidal isomorphism $\eta:\Psi\iso\Phi$. Define $\Psi$ by
\[\Psi(x)=\left\{\begin{array}{ll}u_2 & x=u_1 \\ \Phi(x) & x\neq u_1.\end{array}\right. \]
Given a morphism $x\xrightarrow{f}y$ in $\sP_1$, $\Psi(f)=\Phi(f)$ if neither $x$ nor $y$ is $u_1$, and if either (or both) is $u_1$, then one must compose or precompose with the given isomorphism $\Phi_u:u_2\iso\Phi(u_1)$.

The functor $\Psi:\sP_1\to\sP_2$ is certainly strictly unital. The structure morphism $\Psi(x)\oplus\Psi(y)\iso\Psi(x\oplus y)$ is that of $\Phi$ if $x\neq u_1\neq y$. If $x=u_1$, we take
\[\Psi(u_1)\oplus\Psi(y)=u_2\oplus\Psi(y)=\Psi(y)=\Psi(u_1\oplus y)\]
and similarly if $y=u_1$. The coherence condition for $\Psi$, namely that for all $x,y,z\in\sP_1$, the diagram

\centerline {\xymatrix {
\Psi(x)\oplus\Psi(y)\oplus\Psi(z) \ar[r] \ar[d] & \Psi(x\oplus y)\oplus \Psi(z) \ar[d] \\
\Psi(x)\oplus\Psi(y\oplus z) \ar[r] & \Psi(x\oplus y\oplus z)
}}

\noindent commutes, is then easily verified.

Finally, it remains to specify a monoidal isomorphism $\eta:\Psi\iso\Phi$. For the component of $\eta$ at $x\in\sP_1$, we take either $\Phi_u$ or $1_{\Phi(x)}$, depending on whether $x=u_1$ or not. Because of the way in which we defined $\Psi$ on morphisms, this specifies a natural isomorphism. That $\eta$ is monoidal now follows from the unit coherence conditions for $\Phi$.
\end{proof}

Combining these two results gives the following.

\begin{cor} The inclusion \weakFunctor\ $\mathbf{Perm}_u\into\mathbf{SymMon}$ is a biequivalence of \weakNames.
\end{cor}

\section{Coherence for \weakNames, version $II$}\label{CohVerIISect}

\begin{prop} Let $\mathbf{B}$ be a bicategory and $\mathbf{C}$ be an \weakName. Then the bicategory $\mathbf{Bicat}(\mathbf{B},\mathbf{C})$ of functors, strong transformations, and modifications is an \weakName. If $\mathbf{C}$ is moreover a \strictName, then so is $\mathbf{Bicat}(\mathbf{B},\mathbf{C})$
\end{prop}

\begin{proof}
Let us write $\mathbf{Bicat}_{\mathbf{B},\mathbf{C}}$ for $\mathbf{Bicat}(\mathbf{B},\mathbf{C})$. 
For each pair of functors $F,G:\mathbf{B}\to\mathbf{C}$ of bicategories, we must provide a symmetric monoidal structure on the category $\mathbf{Bicat}_{\mathbf{B},\mathbf{C}}(F,G)$. Let $\eta,\sigma:F\Rightarrow G$ be strong transformations. For every $1$-cell $f:b\to c$ in $\mathbf{B}$, we define $(\eta\oplus\sigma)_b=\eta_b\oplus\sigma_b$ and set
\[(\eta\oplus\sigma)_f:Gf\circ(\eta\oplus\sigma)_b\Rightarrow(\eta\oplus\sigma)_c\circ Ff\]
to be the composite
\[\xymatrix{ Gf\circ(\eta\oplus\sigma)_b & Gf\circ\eta_b\oplus Gf\circ\sigma_b \ar@{=>}_\iso[l] \ar@{=>}[r]^{\eta_f\oplus\sigma_f} & \eta_c\circ Ff\oplus \sigma_c\oplus Ff \ar@{=>}[r]^{\iso} & (\eta\oplus\sigma)_c\circ Ff }. \]
That the above, together with the associativity, unit, and symmetry isomorphisms inherited from $\mathbf{C}(Fb,Gb)$, defines a symmetric monoidal category structure on $\mathbf{Bicat}_{\mathbf{B},\mathbf{C}}(F,G)$ follows immediately since $\mathbf{C}(Fb,Gb)$ is symmetric monoidal.

The composition of $1$-cells in $\mathbf{Bicat}_{\mathbf{B},\mathbf{C}}$ is defined from the composition in $\mathbf{C}$, and the symmetric monoidal structures also come from $\mathbf{C}$. The axioms for a \weakName\ for $\mathbf{Bicat}_{\mathbf{B},\mathbf{C}}$ thus follow immediately from those for $\mathbf{C}$.

If $\mathbf{C}$ is a \strictName, the \strictName \ axioms for $\mathbf{Bicat}_{\mathbf{B},\mathbf{C}}$ also follow immediately since $\mathbf{Bicat}_{\mathbf{B},\mathbf{C}}$ is already a (strict) $2$-category for any $2$-category $\mathbf{C}$.
\end{proof}

\begin{cor} Let $\mathbf{B}$ and $\mathbf{C}$ be \weakNames. Then $\weakMathName(\mathbf{B},\mathbf{C})$ is an \weakName. If $\mathbf{C}$ is moreover a \strictName, then so is the category $\weakMathName(\mathbf{B},\mathbf{C})$.
\end{cor}

\begin{proof} 
The only thing left to check is that the sum of two monoidal transformations is also monoidal. This uses that the distributivity isomorphisms (in $\mathbf{C}$) are monoidal isomorphisms.
\end{proof}


We are now ready to give our second proof of strictification.  

\begin{proof}[Second proof of Theorem~\ref{MainThm}]
The idea is that we can use a Yoneda trick:
\[\mathbf{B}\into \weakMathName(\mathbf{B}^{op},\mathbf{SymMon})\to \weakMathName(\mathbf{B}^{op},\mathbf{Perm}_u).\]
Recall (\cite{Lein}) that, for any bicategory $\mathbf{B}$, the Yoneda functor $Y:\mathbf{B}\into\mathbf{Bicat}(\mathbf{B}^{op},\mathbf{CAT})$ sends a $0$-cell $b$ to the functor 
\[\mathbf{B}(-,b):\mathbf{B}^{op}\longrightarrow\mathbf{CAT}\]
 it represents. For a $1$-cell $f:b\to b'$ in $\mathbf{B}$, $Y(f):\mathbf{B}(-,b)\to\mathbf{B}(-,b')$ is the transformation given by composition with $f$. Similarly, given a $2$-cell $\sigma:f\Rightarrow f'$, $Y(\sigma)$ is the resulting modification $\sigma_*:f\circ-\Rrightarrow f'\circ-$.

\begin{prop} The Yoneda embedding $Y:\mathbf{B}\into\mathbf{Bicat}(\mathbf{B}^{op},\mathbf{CAT})$ factors through a local equivalence of \weakNames\ 
$\mathbf{B}\into \weakMathName(\mathbf{B}^{op},\mathbf{SymMon})$.
\end{prop}

\begin{proof} We prove this in a series of steps.

\begin{lemma}The Yoneda embedding $Y:\mathbf{B}\into\mathbf{Bicat}(\mathbf{B}^{op},\mathbf{CAT})$ factors through a functor of bicategories $\mathbf{B}\into\mathbf{Bicat}(\mathbf{B}^{op},\mathbf{SymMon})$.
\end{lemma}

\begin{proof}
Fix a $0$-cell $b_0\in\mathbf{B}$. Then $Y(b_0)(b)=\mathbf{B}(b,b_0)\in\mathbf{SymMon}$. Moreover, for any $1$-cell $f:b\to c$, the composition with $f$ functor $\mathbf{B}(c,b_0)\xrightarrow{\circ f}\mathbf{B}(b,b_0)$ is strong symmetric monoidal since $\mathbf{B}$ is a \weakName. Finally, if $\eta:f\Rightarrow g$ is a $2$-cell in $\mathbf{B}$, we need to show that the natural transformation $(-)\circ f\Rightarrow(-)\circ g:\mathbf{B}(c,b_0)\to\mathbf{B}(b,b_0)$ is monoidal. But this follows from naturality of the isomorphism $h_1\circ(-)\oplus h_2\circ(-)\iso (h_1\oplus h_2)\circ(-)$ of functors $\mathbf{B}(b,c)\to\mathbf{B}(b,b_0)$. 

Fix a $1$-cell $f_0:b_0\to b_1$ in $\mathbf{B}$. Then for each $b\in\mathbf{B}$, the induced functor $\mathbf{B}(b,b_0)\xrightarrow{f\circ}\mathbf{B}(b,b_1)$ is strong symmetric monoidal by the definition of a \weakName. The requirement that the associativity isomorphism in $\mathbf{B}$ is monoidal makes $\mathbf{B}(-,f_0)$ a transformation of functors to $\mathbf{SymMon}$ (i.e., the required natural transformations are monoidal).

Fix a $2$-cell $\sigma:f_0\Rightarrow g_0:b_0\to b_1$ in $\mathbf{B}$. Then naturality of the distributivity isomorphisms in $\mathbf{B}$ ensures that for each $b\in\mathbf{B}$ the natural transformation
\[\sigma_*:(f_0\circ)\Rightarrow(g_0\circ):\mathbf{B}(b,b_0)\to\mathbf{B}(b,b_1)\]
is monoidal.

Finally, we must show that for every $b_0\xrightarrow{f_0}b_1\xrightarrow{f_1}b_2$ in $\mathbf{B}$, the isomorphisms $Y(f_1)\circ Y(f_0)\iso Y(f_1\circ f_0)$ and $1_{Y(b_0)}\iso Y(1_{b_0})$ are $2$-cells in $\mathbf{Bicat}(\mathbf{B}^{op},\mathbf{SymMon})$. But this is given by the condition that the associativity and unit isomorphisms in $\mathbf{B}$ are monoidal.
\renewcommand{\qedsymbol}{$\blacksquare$}
\end{proof}

\begin{lemma}The functor of bicategories $\mathbf{B}\into\mathbf{Bicat}(\mathbf{B}^{op},\mathbf{SymMon})$ is an \weakFunctor.
\end{lemma}

\begin{proof}
For each $b_0,b_1\in\mathbf{B}$,  we must show that 
\[Y:\mathbf{B}(b_0,b_1)\to\mathbf{Bicat}_{\mathbf{B}^{op},\mathbf{SymMon}}(\mathbf{B}(-,b_0),\mathbf{B}(-,b_1))\]
is strong symmetric monoidal. The desired isomorphism $Y(f_0)\oplus Y(f_1)\iso Y(f_0\oplus f_1)$ is given by a (right) distributivity isomorphism. 

That the natural isomorphism $Y(f\circ g)\iso Y(f)\circ Y(g)$ is monoidal (in $f$) is given by the condition that the  associativity isomorphism in $\mathbf{B}$ is monoidal.
\renewcommand{\qedsymbol}{$\blacksquare$}
\end{proof}

\begin{lemma}The image of the \weakFunctor\ $\mathbf{B}\into\mathbf{Bicat}(\mathbf{B}^{op},\mathbf{SymMon})$ lies in the sub-$2$-category $\weakMathName(\mathbf{B}^{op},\mathbf{SymMon})$.
\end{lemma}

\begin{proof}
Fix a $0$-cell $b_0\in\mathbf{B}$. We must show that the functor of bicategories $Y(b_0):\mathbf{B}^{op}\to\mathbf{SymMon}$ is in fact an \weakFunctor. Thus for each $b,c\in\mathbf{B}$, the functor
\[\mathbf{B}^{op}(b,c)=\mathbf{B}(c,b)\to\mathbf{SymMon}(\mathbf{B}(b,b_0),\mathbf{B}(c,b_0))\]
must be strong symmetric monoidal. But this is part of the definition of an \weakName. That the natural isomorphism 
\[\bigl( Y(b_0)(f)\bigr) \circ \bigl(Y(b_0)(-)\bigr) \iso Y(b_0)\bigl( f\circ-\bigr)\]
is monoidal in $f$ follows from the fact that the associativity isomorphisms in $\mathbf{B}$ are monoidal.

Fix a $1$-cell $f_0:b_0\to b_1$ in $\mathbf{B}$. We must show that $Y(f_0):Y(b_0)\Rightarrow Y(b_1)$ is a monoidal transformation. This follows from the condition that the associativity isomorphism $(f_0\circ h)\circ -\iso f_0\circ(h\circ -)$ is monoidal.

The sub-$2$-category $\weakMathName(\mathbf{B}^{op},\mathbf{SymMon})\subset\mathbf{Bicat}(\mathbf{B}^{op},\mathbf{SymMon})$ is locally full, so there is no condition to check for $2$-cells. 
\renewcommand{\qedsymbol}{$\blacksquare$}
\end{proof}

Finally, we must verify that $Y:\mathbf{B}\to\weakMathName(\mathbf{B}^{op},\mathbf{SymMon})$ is a local equivalence. In other words, we must show that for every pair of $0$-cells $b_0,b_1\in\mathbf{B}$, the strong symmetric monoidal functor $Y:\mathbf{B}(b_0,b_1)\to\weakMathName_{\mathbf{B}^{op},\mathbf{SymMon}}(Y(b_0),Y(b_1))$ is an equivalence. 

But we know already that the composite
\[\mathbf{B}(b_0,b_1)\to\weakMathName_{\mathbf{B}^{op},\mathbf{SymMon}}(Y(b_0),Y(b_1))\to\mathbf{Bicat}_{\mathbf{B}^{op},\mathbf{CAT}}(Y(b_0),Y(b_1))\]
is an equivalence, and the second functor is clearly faithful (\ref{LocalEquiv}). It follows that both functors in this composition are equivalences.
\end{proof}

\begin{lemma} Let $\mathbf{B}$, $\mathbf{C}$, and $\mathbf{D}$ be \weakNames. Any biequivalence \mbox{$J:\mathbf{C}\to\mathbf{D}$} induces a biequivalence $J_*:\weakMathName(\mathbf{B},\mathbf{C})\to \weakMathName(\mathbf{B},\mathbf{D})$.
\end{lemma}

Combining these two results with a quasi-inverse $J$ to the biequivalence $\mathbf{Perm}_u\into\mathbf{SymMon}$ gives a local equivalence
\[\mathbf{Y}:\mathbf{B}\to \weakMathName(\mathbf{B}^{op},\mathbf{Perm}_u)\]
that is injective on $0$-cells. Since $\weakMathName(\mathbf{B}^{op},\mathbf{Perm})$ is a \strictName, the full sub-$2$-category containing only $0$-cells in the image of $\mathbf{Y}$ is a \strictName\ equipped with a biequivalence from $\mathbf{B}$.
\renewcommand{\qedsymbol}{$\blacksquare$}
\end{proof}

\section{The single $0$-cell case}\label{SingZeroSect}

As we have discussed above, a single $0$-cell \weakName\ is a bimonoidal category, and a single $0$-cell \strictName\ is a strict bimonoidal category, a strong form of the ring categories of \cite{EM}. In this final section, we discuss the strictification of bimonoidal categories given in \S\ref{CohVerIISect}.

Thus let $\mathbf{M}$ be a bimonoidal category, thought of as a single $0$-cell \weakName\ $\mathbf{BM}$. The strictification of $\mathbf{M}$ outlined above is the endomorphism (strict monoidal) category of the \weakFunctor\ 
\[\mathbf{BM}^{op}\xrightarrow{\mathbf{Y}(*)}\mathbf{SymMon}\xrightarrow{J}\mathbf{Perm}_u.\]
Thus the objects of the strictification $\cP$ of $\mathbf{M}$ (really the $1$-cells of the strictification of $\mathbf{BM}$) are the strong monoidal transformations $J\mathbf{Y}(*)\Rightarrow J\mathbf{Y}(*)$. Such a transformation consists of a strong monoidal, strictly unital functor $\Phi:J(\mathbf{M})\to J(\mathbf{M})$ and, for each $m\in\mathbf{M}$, an isomorphism
\[ \Lambda^\Phi_m:\Phi\circ J(-\otimes m)\iso J(-\otimes m)\circ \Phi \]
that is natural and monoidal in $m$.

The morphisms $(\Phi,\Lambda^\Phi)\to(\Psi,\Lambda^\Psi)$ of $\cP$ ($2$-cells of the strictification of $\mathbf{BM}$) are modifications. This consists of a monoidal transformation $\sigma:\Phi\Rightarrow\Psi$ such that for each $m\in\mathbf{M}$, the diagram
\[\xymatrix{
\Phi\circ J(-\otimes m) \ar[r]^{\sigma\circ 1} \ar[d]_{\Lambda^\Phi_m} & \Psi\circ J(-\otimes m) \ar[d]^{\Lambda^\Psi_m} \\
J(-\otimes m)\circ \Phi \ar[r]_{1\circ\sigma} & J(-\otimes m)\circ \Psi
}\]
commutes.

The additive monoidal structure is given by taking sums of strong monoidal, strictly unital functors, as in Example~\ref{PermuEG}. The unit for this monoidal structure is the constant functor at $0\in J(\mathbf{M})$. The multiplicative monoidal structure is given by composition, and the unit for this monoidal structure is the identity functor of $J(\mathbf{M})$.

In case the additive monoidal structure on $\mathbf{M}$ is already strict, the effect of $J$ is simply to make the right multiplication functors $-\otimes m$ strictly unital. Let us write $(-\tilde{\otimes}m)$ for $J(-\otimes m)$. That is,
\[
n\tilde{\otimes}m=\left\{\begin{array}{cc} n \otimes m & n\neq 0 \\ 0 & n=0.\end{array}\right.
\]
Then for $\Phi:M\to M$ strong monoidal and strictly unital as above, the data of isomorphisms
\[\Lambda_m^\Phi:\Phi\circ(-\tilde{\otimes}m)\iso(-\tilde{\otimes}m)\circ\Phi\]
for all $m$ is equivalent to a single isomorphism
\[\lambda^\Phi:\Phi\iso \Phi(1)\tilde{\otimes}-=\left\{\begin{array}{cc} \Phi(1) \otimes - & \Phi(1)\neq 0 \\ 0 & \Phi(1)=0.\end{array}\right.\]
Morphisms $\sigma:(\Phi,\lambda^\Phi)\Rightarrow(\Psi,\lambda^\Psi)$ must make the diagram
\[\xymatrix{
\Phi(n) \ar[r]^{\sigma_n} \ar[d]_{\lambda^\Phi_n} & \Psi(n) \ar[d]^{\lambda^\Psi_n} \\
\Phi(1)\tilde{\otimes}n \ar[r]_{\sigma_1\otimes 1} & \Psi(1)\tilde{\otimes} n
}\]
commute and are therefore determined by their value at $1$. From this it follows that the evaluation at $1$ functor $ev_1:\cP\to\mathbf{M}$ is an equivalence of bimonoidal categories.


\bibliographystyle{plain}
\bibliography{Perm2CatRefs}

\begin{thebibliography}{1}

\bibitem{EM}
A.~D. Elmendorf and M.~A. Mandell.
\newblock Rings, modules, and algebras in infinite loop space theory.
\newblock {\em Adv. Math.}, 205(1):163--228, 2006.

\bibitem{Isb}
John~R. Isbell.
\newblock On coherent algebras and strict algebras.
\newblock {\em J. Algebra}, 13:299--307, 1969.

\bibitem{Lap}
Miguel~L. Laplaza.
\newblock Coherence for distributivity.
\newblock In {\em Coherence in categories}, pages 29--65. Lecture Notes in
  Math., Vol. 281. Springer, Berlin, 1972.

\bibitem{Lein}
T.~Leinster.
\newblock {Basic bicategories}.
\newblock {\em Arxiv preprint math/9810017}, 1998.

\bibitem{MQRT}
J.~Peter May.
\newblock {\em {$E\sb{\infty }$}-ring spaces and {$E\sb{\infty }$}-ring
  spectra}.
\newblock Lecture Notes in Mathematics, Vol. 577. Springer-Verlag, Berlin,
  1977.
\newblock With contributions by Frank Quinn, Nigel Ray, and J{\o}rgen
  Tornehave.

\end{thebibliography}

\end{document}